\documentclass{conm-p-l}

\usepackage{amssymb,amsmath,color,psfrag,graphicx}

\numberwithin{equation}{section}

\newtheorem{thm}[equation]{Theorem}

\newtheorem{cor}[equation]{Corollary}
\newtheorem{lem}[equation]{Lemma}
\newtheorem{prop}[equation]{Proposition}

\theoremstyle{definition}
\newtheorem{defn}[equation]{Definition}

\theoremstyle{remark}
\newtheorem{rem}[equation]{Remark}

\newcommand{\thmref}[1]{Theorem~\ref{#1}}
\newcommand{\propref}[1]{Proposition~\ref{#1}}
\newcommand{\lemref}[1]{Lemma~\ref{#1}}
\newcommand{\corref}[1]{Corollary~\ref{#1}}

\newcommand{\figref}[1]{Figure~\ref{#1}}
\newcommand{\secref}[1]{Section~{\bf\ref{#1}}}

\renewcommand\a{{\alpha}}

\newcommand\B{{\mathcal B}}
\renewcommand\b{{\beta}}

\newcommand\ch{\check}
\newcommand\cw{{\curlywedge}}

\newcommand\D{{\Delta}}
\renewcommand\d{{\delta}}

\renewcommand\deg{\operatorname{\mathsf{deg}}}
\newcommand\diag{\operatorname{\mathsf{diag}}}
\newcommand\du{\mathop{\uparrow\!\downarrow}}

\newcommand\F{{\Phi}}
\newcommand\f{{\varphi}}
\newcommand\ft[1]{{\parbox{110truemm}{#1}}}

\newcommand\g{\gamma}

\newcommand\GG{{\mathcal G}}

\renewcommand\H{{\mathbf H}}
\newcommand\h{{\mathfrak h}}

\renewcommand\l{\ell}
\newcommand\la{\lambda}

\newcommand\Mu{{\mathrm M}}

\newcommand\Nu{{\boldsymbol \#}}

\renewcommand\o{{\omega}}

\newcommand\ov{\overline}

\renewcommand\P{\mathbf P}

\newcommand\pt{\partial}

\newcommand\R{{\mathbb R}}

\newcommand\si{\sigma}

\newcommand\X{{\mathcal X}}
\newcommand\x{{\boldsymbol x}}

\newcommand\Z{\mathbb Z}

\begin{document}

\title[Random horospheric products]{Random walks on random
horospheric products}

\author{Vadim A. Kaimanovich}

\thanks{The research of the first author was undertaken, in part, thanks to funding
from the Canada Research Chairs program. Support from ERC grants 208091 GADA
and 257110 RaWG  is also gratefully acknowledged.}

\address{Department of Mathematics and Statistics, University of Ottawa,
585 King Edward, Ottawa ON, K1N 6N5, Canada}

\email{vkaimano@uottawa.ca, vadim.kaimanovich@gmail.com}

\author{Florian Sobieczky}

\address{Department of Mathematics, University of Colorado at Boulder, Campus Box 395 }

\email{florian.sobieczky@colorado.edu}

\begin{abstract}
By developing the entropy theory of random walks on equivalence relations and
analyzing the asymptotic geometry of horospheric products we describe the
Poisson boundary for random walks on random horospheric products of trees.
\end{abstract}



\maketitle

\thispagestyle{empty}

\section*{Introduction}

\emph{Horospheric products of trees} were first introduced in the work of
Diestel and Leader \cite{Diestel-Leader01} in an attempt to answer a question
of Woess \cite{Woess91} on existence of vertex-transitive graphs not
quasi-isometric to Cayley graphs. Although the fact that the Diestel--Leader
graphs indeed provide such an example was only recently proved by Eskin,
Fisher and Whyte \cite{Eskin-Fisher-Whyte07}, in the meantime the construction
of Diestel and Leader attracted a lot of attention because of its numerous
interesting features (see \cite{Woess05,Bartholdi-Neuhauser-Woess08} and the
references therein). For instance, as it was observed by Woess, the
horospheric product of two homogeneous trees of the same degree $p+1$ is
isomorphic to the Cayley graph of the lamplighter group (the wreath product
$\Z\wr\Z_p$) with respect to an appropriate generating set. This observation
was the starting point for Bartholdi and Woess \cite{Bartholdi-Woess05} who
showed that, along with lamplighter groups, horospheric products of
homogeneous trees (not necessarily of the same degree!) provide one of very
few examples of infinite graphs, for which all spectral invariants can be
exhibited in an absolutely explicit form (and, in addition, the spectrum
happens to be pure point).

The construction of horospheric products is very natural from a geometrical
viewpoint. Namely, by choosing a point $\g$ on the boundary $\pt T$ of an
infinite tree $T$ one converts it into the genealogical tree generated by the
``mythological progenitor''~$\g$. The \emph{Busemann cocycle} $\b_\g$ on $T$
can be interpreted as the signed ``generations gap'' in the genealogical tree:
its level sets are the ``generations'' in $T$ as seen from $\g$. Given another
pointed at infinity tree $(T',\g')$, the horospheric product (or, rather,
products) of $(T,\g)$ and $(T',\g')$ are then the level sets of the aggregate
cocycle $\b_\g+\b_{\g'}$ on $T\times T'$, and the Busemann cocycles determine
a natural \emph{``height cocycle''} on individual horospheric products. It is
important for what follows that each horospheric product is endowed with two
boundaries (``lower'' and ``upper'') isomorphic to the punctured boundaries
$\pt T\setminus\{\g\}$ and $\pt T'\setminus\{\g'\}$ of the trees $T$ and~$T'$.

In the previous paper \cite{Kaimanovich-Sobieczky10} we considered the problem
of \emph{stochastic homogenization} for horospheric products. The approach
that we used there was an implementation of the ideas from
\cite{Kaimanovich03a}: to consider random graphs as leafwise graphs of an
appropriate \emph{graphed equivalence relation}; stochastic homogenization
means that there is a probability measure invariant with respect to this
relation. Here we are continuing to apply the ideas from \cite{Kaimanovich03a}
to random horospheric products by looking at random walks on them, or, in view
of the aforementioned reduction, at random walks along classes of graphed
equivalence relations, leafwise graphs of which are horospheric products.

The problem we address is that of the boundary behaviour of such leafwise
random walks, more precisely, the problem of identification of their
\emph{Poisson boundaries}. In the case of isotropic random walks on
horospheric products of homogeneous trees this problem (or, actually, even the
more general problem of describing the Martin boundary) was solved by
Brofferio and Woess \cite{Woess05,Brofferio-Woess05,Brofferio-Woess06}.
However, their approach (as is almost always the case with the Martin
boundary) heavily depends on explicit estimates of the Green kernel only
possible for highly symmetrical Markov chains. Following
\cite{Kaimanovich03a}, instead of this we use the \emph{entropy theory}.
Originally developed for dealing with the Poisson boundary of random walks on
groups (see \cite{Kaimanovich-Vershik83,Kaimanovich00a} and the references
therein), it is actually applicable in all situations when there is an
appropriate probability path space endowed with a measure preserving time
shift, in particular for random walks on equivalence relations in the presence
of a global stationary measure. In this setup the entropy theory was already
outlined by the first author in \cite{Kaimanovich98,Kaimanovich03a}; here we
give a more detailed exposition.

In the group case the entropy theory produces not only the entropy criterion
of \emph{boundary triviality}, but also very efficient geometrical conditions
for \emph{identification of the Poisson boundary} (``ray'' and ``strip''
approximations). Both these conditions readily carry over to random walks
along classes of graphed equivalence relations as well. In order to apply them
to horospheric products we establish the necessary geometrical ingredients.
Namely, we completely characterize \emph{geodesics} in horospheric products
and give necessary and sufficient conditions for a sequence of points to be
\emph{regular}, i.e., to follow a geodesic with a sublinear deviation.

As a consequence we establish our \textbf{main result} (\thmref{th:main}),
according to which the Poisson boundary of a random walk on an equivalence
relation graphed by horospheric products in the presence of a global
stationary probability measure is completely determined by the \emph{height
drift} $h$ (the expectation of the height cocycle): if $h=0$, then a.e.\
leafwise Poisson boundary is trivial, whereas if $h\neq 0$ then a.e.\ leafwise
Poisson boundary coincides with the corresponding (lower or upper, depending
on the sign of $h$) boundary of the underlying horospheric product endowed
with the corresponding limit (hitting) distribution. This description is in
perfect keeping with the situation for horospheric products of homogeneous
trees \cite{Woess05} or for lamplighter groups \cite{Kaimanovich91} (whose
Cayley graphs for an appropriate choice of generators are horospheric products
of homogeneous trees of the same degree \cite{Woess05}).

The main result implies that for \emph{reversible} random walks on random
horospheric products the leafwise Poisson boundaries are almost surely
trivial, because in this situation the height drift (being the expectation of
an additive cocycle) vanishes. This is the case for simple random walks on
stochastically homogeneous horospheric products considered in
\cite{Kaimanovich-Sobieczky10}, in particular, for horospheric products of
augmented Galton--Watson trees with the same offspring expectation.

On the other hand, although already lamplighter groups and horospheric
products of homogeneous trees readily provide examples of random walks on
horospheric products with \emph{non-zero height drift}, it would be
interesting to have more ``probabilistically natural'' examples of this kind.

It is worth mentioning in this respect that the \emph{homesick simple random
walk} with an integer parameter $d$ on a pointed at infinity tree $T$ can be
interpreted as the projection of the usual simple random walk on the
horospheric product of $T$ and the homogeneous tree of degree $d+1$ (usually
``homesickness'' is defined with respect to a reference point inside the
graph, e.g., see \cite{Lyons-Pemantle-Peres96}, but this definition in an
obvious way adapts to pointed at infinity trees as well). For usual simple
random walks on random Galton--Watson trees existence of a linear rate of
escape was established in \cite{Lyons-Pemantle-Peres95} by using an explicit
stationary measure on the space of trees. In the homesick case, although a
linear rate of escape still exists \cite{Lyons-Pemantle-Peres96a}, no such
construction is known.

Yet another link between horospheric products and homesick random walks worth
further investigation is provided by a rather unexpected behavior of the rate
of escape of homesick random walks on the lamplighter group exhibited in
\cite{Lyons-Pemantle-Peres96a} (although homesickness in
\cite{Lyons-Pemantle-Peres96a} is defined with respect to the standard
generating set rather than the one whose Cayley graph is a horospheric
product).

Let us finally mention that our results (with rather straightforward
modifications) carry over to horospheric products with more than two
multipliers which were introduced in \cite[p.~356]{Kaimanovich-Woess02} and
further studied in \cite{Bartholdi-Neuhauser-Woess08}.

\medskip

The paper has the following structure. In \secref{sec:asgeo} we study the
\emph{asymptotic geometry of individual horospheric products}. After reminding
the necessary definitions concerning trees (\secref{sec:trees}) and their
horospheric products (\secref{sec:products}), in \secref{sec:geod} we reprove
Bertacchi's formula \cite{Bertacchi01} for the distance in a horospheric
product (\propref{pr:distance}). Our argument is somewhat different and
provides an explicit description of geodesic segments in horospheric products,
on the base of which we further describe geodesic rays and bilateral geodesics
(\propref{pr:ray} and \propref{pr:bil}, respectively). In \secref{sec:reg} we
give criteria for a sequence of points in a horospheric product to be regular
(\thmref{th:ray}). Finally, in \secref{sec:bdry} we discuss boundaries of
horospheric products.

\secref{sec:random} contains the probabilistic part of our arguments. We begin
by reminding the basic definitions concerning graphed equivalence relations
and random graphs (\secref{sec:eq}). In \secref{sec:rweq} we discuss Markov
chains along classes of an equivalence relation endowed with a quasi-invariant
measure; the exposition here is based on \cite{Kaimanovich98}. We express the
action of the corresponding Markov operator on measures in terms of the
leafwise transition probabilities and the Radon--Nikodym cocycle of the
equivalence relation (\propref{pr:deriv}) and give a necessary and sufficient
condition for stationarity of a measure on the state space
(\corref{cor:stat}). In particular, an invariant measure of a graphed
equivalence relation becomes stationary for the leafwise simple random walk
after multiplication by the density equal to the vertex degree function
(\corref{cor:simple}).

In \secref{sec:entr} we develop the \emph{entropy theory for random walks on
equivalence relation}. The exposition here follows the outlines given in
\cite{Kaimanovich98,Kaimanovich03a} and is completely parallel to the entropy
theory for random walks in random environment on groups \cite{Kaimanovich90a}
(which, in turn, was inspired by the case of the usual random walks on groups
\cite{Kaimanovich-Vershik83}). First we prove that the leafwise tail and
Poisson boundaries coincide $\P_x$ -- mod 0 for a.e.\ initial point $x$
(\thmref{th:02}), after which we define the \emph{asymptotic entropy} $\h$ and
prove that the leafiwse tail ($\equiv$ Poisson) boundaries are a.e.\ trivial
if and only if $\h=0$ (\thmref{th:entr}). By passing to an appropriate
\emph{boundary extension} of the original equivalence relation
\cite{Kaimanovich05a}, \thmref{th:entr} is also applicable to the problem of
description of non-trivial Poisson boundaries of leafwise Markov chains.
Indeed, a quotient of the Poisson boundary is \emph{maximal} (i.e., coincides
with the whole Poisson boundary) if and only if for almost all conditional
chains determined by the points of this quotient the Poisson boundary is
trivial. Thus, the criterion from \thmref{th:entr} allows one to carry over
the \emph{ray} and the \emph{strip} criteria used for identification of the
Poisson boundary in the group case \cite{Kaimanovich00a} to the setup of
random walks along classes of graphed equivalence relations.

Finally, in \secref{sec:final} we formulate and prove the main result of the
present paper: the aforementioned description of Poisson boundaries of random
walks along random horospheric products (\thmref{th:main}).

\section{Asymptotic geometry of horospheric products} \label{sec:asgeo}

\subsection{Trees} \label{sec:trees}

We begin by recalling that a \emph{tree} is a connected graph without cycles.
Any two vertices $x,y$ in a tree $T$ can be joined with a unique segment
$[x,y]$ which is \emph{geodesic} with respect to the standard graph distance
$d$. Throughout the paper we will only be considering trees ``without
leaves'', i.e., such that the degree of any vertex is at least 2.

Any locally finite tree~$T$ has a natural \emph{compactification} $\ov T = T
\sqcup \pt T$ obtained in the following way: a sequence of vertices $x_n$
which goes to infinity in $T$ converges in this compactification if and only
if for a certain ($\equiv$ any) reference point $o\in T$ the geodesic segments
$[o,x_n]$ converge pointwise. Thus, for any reference point $o\in T$ the
\emph{boundary} $\pt T$ can be identified with the space of geodesic rays
issued from $o$ (and endowed with the topology of pointwise convergence).
There are many other equivalent descriptions of the boundary $\pt T$ (and of
the compactification~$\ov T$), in particular, as the \emph{space of ends} of
$T$ and as the \emph{hyperbolic boundary} of $T$.

A tree $T$ with a distinguished boundary point $\g\in\pt T$ is called
\emph{pointed at infinity} ($\equiv$ remotely rooted; in the terminology of
Cartier \cite{Cartier72} the point $\g$ is called a ``mythological
progenitor''). We shall use the notation $\pt_\odot T=\pt T\setminus\{\g\}$
for the \emph{punctured boundary} of a pointed at infinity tree $(T,\g)$. A
triple $T_o^\g=(T,o,\g)$ with $o\in T$ and $\g\in\pt T$ is a \emph{rooted tree
pointed at infinity}.

Any two geodesic rays converging to the same boundary point eventually meet,
so that any boundary point $\g\in\pt T$ determines the associated additive
$\Z$-valued \emph{Busemann cocycle} on $T$. It is defined as
\begin{equation} \label{eq:Bu}
\b_\g(x,y) = d(y,z) - d(x,z) \;,
\end{equation}
where $ z = x \cw_\g y $ is the \emph{confluence} of the geodesic rays
$[x,\g)$ and $[y,\g)$, see \figref{fig:bus}.

\begin{figure}[h]
\begin{center}
     \psfrag{x}[cr][cr]{$x$}
     \psfrag{o}[cr][cr]{$z$}
     \psfrag{y}[cr][cr]{$y$}
     \psfrag{g}[cl][cl]{$\g$}
     \psfrag{d}[cl][cl]{$\pt T$}
          \includegraphics{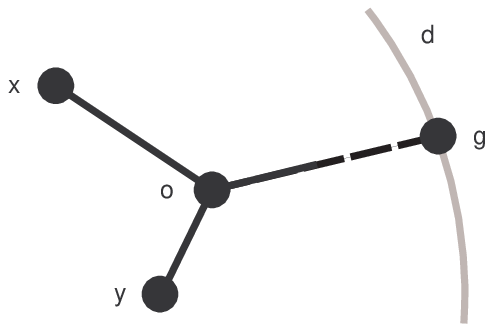}
          \end{center}
          \caption{}
          \label{fig:bus}
\end{figure}

Obviously,
$$
|\b_\g(x,y)| \le d(x,y) \qquad\forall\,x,y\in T,\;\g\in\pt T \;.
$$
The Busemann cocycle can also be defined as
$$
\b_\g(x,y) = \lim_{z\to\g} \bigl[ d(y,z) - d(x,z) \bigr] \;,
$$
so that it is a ``regularization'' of the formal expression $d(y,\g)-d(x,\g)$.
In the presence of a reference point $o\in T$ one can also talk about the
\emph{Busemann function}
$$
b_\g(x) = \b_\g(o,x) \;.
$$
The level sets
$$
H_k = \{x\in T : b_\g(x) = k \}
$$
of the Busemann function ($\equiv$ of the Busemann cocycle) are called
\emph{horospheres} centered at the boundary point $\g$, see \figref{fig:tree}.

\begin{figure}[h]
\begin{center}
     \psfrag{hm}[cl][cl]{$H_{-1}$}
     \psfrag{o}[cr][cr]{$o$}
     \psfrag{h}[cl][cl]{$H_0$}
     \psfrag{g}[cl][cl]{$\g$}
     \psfrag{hp}[cl][cl]{$H_1$}
          \includegraphics{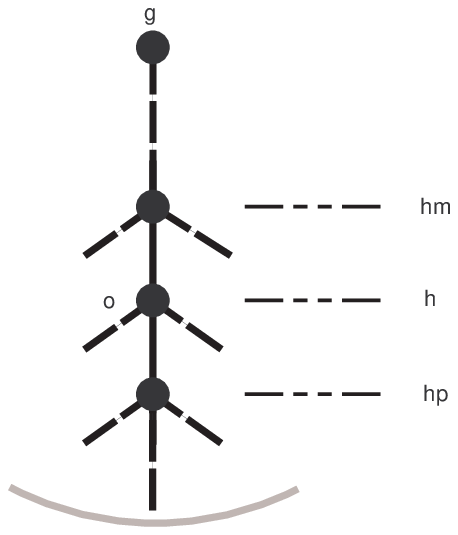}
          \end{center}
          \caption{}
          \label{fig:tree}
\end{figure}

\subsection{Horospheric products} \label{sec:products}

\begin{defn} \label{def:DL}
Let $T=(T,o,\g)$ and $T'=(T',o',\g')$ be two rooted trees pointed at infinity,
and let $b=\b_{\g}(o,\cdot),\; b'=\b_{\g'}(o',\cdot)$ be the corresponding
Busemann functions. The \emph{horospheric product} $T\du T'$ is the graph with
the vertex set
$$
\{(x,x')\in T\times T': b(x)+b'(x')=0 \}
$$
and the edge set
$$
\bigl\{\bigl( (x,x'), (y,y') \bigr) : (x,y) \;\text{and}\;(x',y') \;\text{are
edges in}\; T,T',\;\text{respectively} \bigr\} \;.
$$
\end{defn}

Geometrically one can think about the horospheric products in the following
way \cite{Kaimanovich-Woess02}. Draw the tree $T'$ upside down next to $T$ so
that the respective horospheres $H_k(T)$ and $H_{-k}(T')$ are at the same
level. Connect the two origins $o,o'$ with an elastic spring. It can move
along each of the two trees, may expand or contract, but must always remain
horizontal. The vertex set of $T\du T'$ consists then of all admissible
positions of the spring. From a position $(x,x')$ with $b(x)+b'(x')=0$ the
spring may move downwards to one of the ``sons'' of $x$ and at the same time
to the ``father'' of $x'$, or upwards in an analogous way. Such a move
corresponds to going to a neighbour $(y,y')$ of $(x,x')$, see \figref{fig:dl}.

\begin{figure}[h]
\begin{center}
     \psfrag{o2}[cl][cl]{$o'$}
     \psfrag{o}[cr][cr]{$o$}
     \psfrag{g2}[cl][cl]{$\g'$}
     \psfrag{g}[cl][cl]{$\g$}
          \includegraphics{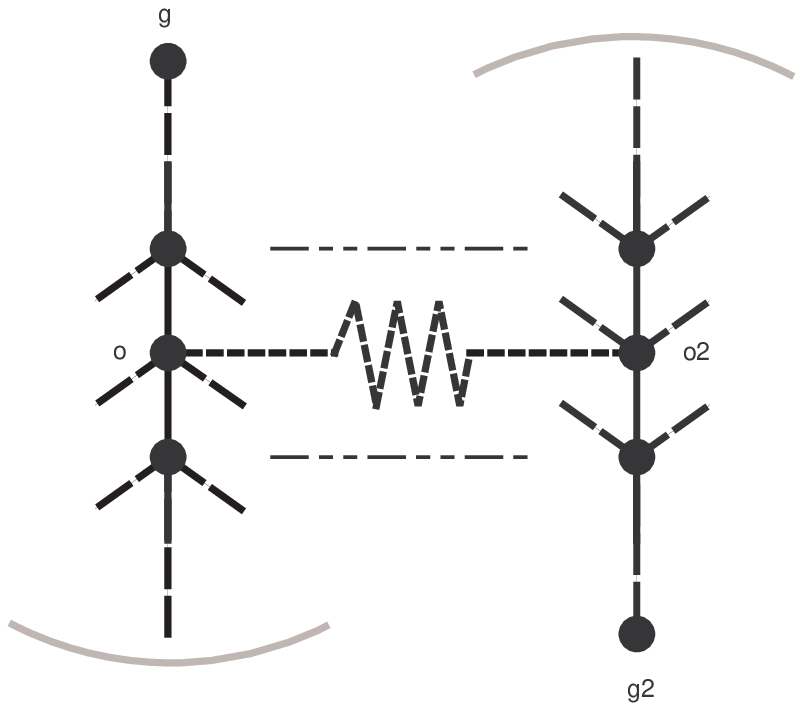}
          \end{center}
          \caption{}
          \label{fig:dl}
\end{figure}

See \cite{Woess05,Bartholdi-Neuhauser-Woess08,Kaimanovich-Sobieczky10} and the
references therein for the historical background and recent works on
horospheric products of trees (aka \emph{Diestel--Leader graphs} or
\emph{horocyclic products}).

We shall use capital letters for denoting points of the horospheric product
$T\du T'$ (so that $X=(x,x')$ with $x\in T,x'\in T'$, etc.). In particular, we
denote by $O=(o,o')$ the reference point in $T\du T'$. The graph $T\du T'$ is
endowed with the \emph{height cocycle}
\begin{equation} \label{eq:height}
\B(X,Y) = \b_\g(x,y) = -\b_{\g'}(x',y') \;.
\end{equation}
For simplicity below we shall use the ``height function'' on $T\du T'$
$$
\ov X = -\B(O,X)
$$
defined in accordance with \figref{fig:dl} (so that the ``higher'' is the
level, the bigger is the value $\ov X$). In the same way, we put $\ov x=-b(x)$
and $\ov{x'}=b'(x')$ for any $x\in X, x'\in X'$, so that
$$
\ov X = \ov x = \ov{x'} \qquad \forall\; X=(x,x')\in T\du T' \;.
$$

\subsection{Geodesic segments and rays} \label{sec:geod}

Before establishing an explicit formula for the graph metric on the
horospheric product $T\du T'$ let us first notice that the sheer existence of
the natural projections of $T\du T'$ onto $T,T'$ and $\Z$ (the latter by the
height function) implies the obvious inequalities
\begin{equation} \label{eq:dineq}
d(x,y),\; d(x',y'),\; |\B(X,Y)| \le d(X,Y)
\end{equation}
for all pairs of points $X=(x,x'),Y=(y,y')\in T\du T'$.

Formula \eqref{eq:dDL} below for the graph metric in $T\du T'$ was first
established by Bertacchi \cite[Proposition~3.1]{Bertacchi01} (although
Bertacchi considered horospheric products of homogeneous trees only, her
arguments are actually valid in the general case as well). We shall give here
a somewhat different argument, which, in particular, allows us to obtain an
explicit description of all geodesics in $T\du T'$.

\begin{prop} \label{pr:distance}
The graph distance in the horospheric product $T\du T'$ is
\begin{equation} \label{eq:dDL}
d(X,Y) = d(x,y) + d(x',y') - |\B(X,Y)|  \;.
\end{equation}
for all $X=(x,x'),\, Y=(y,y') \in T\du T'$.
\end{prop}

\begin{proof}
Let $\F$ be a path joining the points $X$ and $Y$. Then its projection $\f$ to
$T$ (resp., its projection $\f'$ to $T'$) joins $x$ and $y$ (resp., $x'$ and
$y'$). Since $T$ and $T'$ are trees, $\f$ and $\f'$ should pass through all
edges of the geodesics $[x,y]$ and $[x',y']$, respectively. Let $x\cw
y=x\cw_\g y$ and $x'\cw y'=x'\cw_{\g'} y'$ be the confluences of the geodesic
rays $[x,\g),[y,\g)$ and $[x',\g'),[y',\g')$, respectively. The geodesic
$[x,y]$ in $T$ consists of the ascending part $[x,x\cw y]$ (along which the
height increases) and the descending part $[x\cw y,y]$ (along which the height
decreases). In the same way the geodesic $[x',y']$ in $T'$ consists of the
descending part $[x',x'\cw y']$ and the ascending part $[x'\cw y',y']$, cf.
\figref{fig:bus}.

Thus,
\begin{equation} \label{eq:proj}
\ft{\textsf{the projection $\phi$ of $\F$ to $\Z$ (by the height function)
joins the points $\ov X,\ov Y\in\Z$ and contains all the edges (with the
appropriate orientation!) from the oriented segments $\left[\,\ov X,\ov{x\cw
y}\,\right],\; \left[\,\ov{x\cw y},\ov Y\,\right]$ and $\left[\,\ov
X,\ov{x'\cw y'}\,\right],\; \left[\,\ov{x'\cw y'},\ov Y\,\right]$ \;. }}
\end{equation}

These segments do not overlap (if their orientation is taken into account),
except for the oriented segment $\left[\,\ov X,\ov Y\,\right]$ which appears
twice (see \figref{fig:geod}, where $\ov X<\ov Y$).

\begin{figure}[h]
\begin{center}
     \psfrag{x}[cl][cl]{$x$}
     \psfrag{y}[cl][cl]{$y$}
     \psfrag{x1}[cl][cl]{$x'$}
     \psfrag{y1}[cl][cl]{$y'$}
     \psfrag{xy}[cl][cl]{$x\cw y$}
     \psfrag{xy1}[cl][cl]{$x'\cw y'$}
     \psfrag{ox}[cl][cl]{$\ov X$}
     \psfrag{oy}[cl][cl]{$\ov Y$}
     \psfrag{oxy}[cl][cl]{$\ov{x\cw y}$}
     \psfrag{oxy1}[cl][cl]{$\ov{x'\cw y'}$}
          \includegraphics{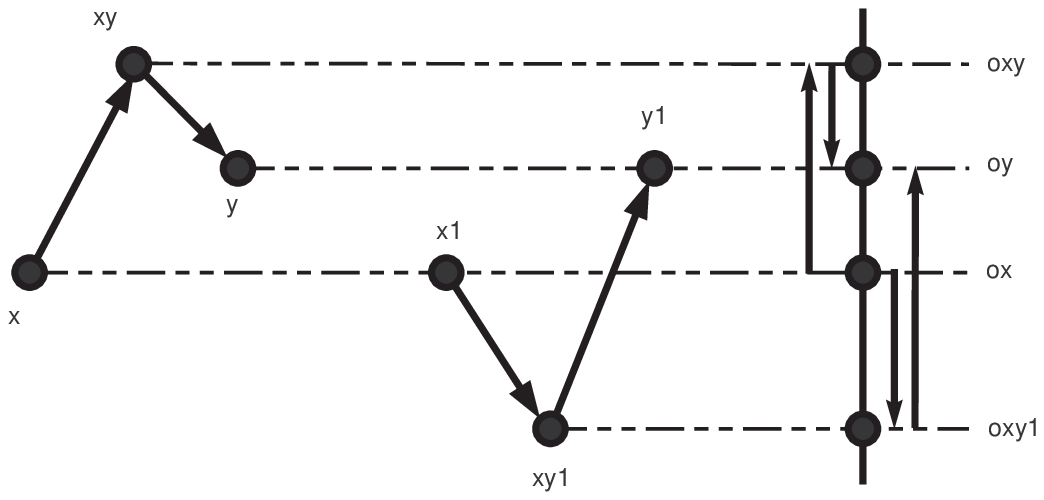}
          \end{center}
          \caption{}
          \label{fig:geod}
\end{figure}

Therefore, the length of $\F$ satisfies the inequality
$$
|\F| \ge \Bigl|\left[\,\ov X,\ov{x\cw y}\,\right]\Bigr| +
\Bigl|\left[\,\ov{x\cw y},\ov Y\,\right]\Bigr| + \Bigl|\left[\,\ov X,\ov{x'\cw
y'}\,\right]\Bigr| + \Bigl|\left[\,\ov{x'\cw y'},\ov Y\,\right]\Bigr| -
\Bigl|\left[\,\ov X,\ov Y\,\right]\Bigr| \;,
$$
the right-hand side of which being precisely the right-hand side of equation
\eqref{eq:dDL}, so that we have proved the inequality
$$
d(X,Y) \ge d(x,y) + d(x',y') - |\B(X,Y)| \;.
$$

\medskip

Now we shall show that paths of length $d(x,y) + d(x',y') - |\B(X,Y)|$ joining
$X$ and $Y$ do exist, and, moreover, we shall explicitly describe all of them.
Let us consider three cases.

\medskip

(i) $\ov X<\ov Y$. Then there exists a unique path $\phi$ in $\Z$ of length
$d(x,y) + d(x',y') - |\B(X,Y)|$ satisfying condition \eqref{eq:proj}. Indeed,
there is only one way to make a path joining $\ov X$ and $\ov Y$ by using (one
time each) all the oriented edges contained in the segments from
\eqref{eq:proj}. This is the path
$$
\phi = \left[\,\ov X, \ov{x'\cw y'}\,\right] \, \left[\,\ov{x'\cw y'},
\ov{x\cw y}\,\right] \, \left[\,\ov{x\cw y}, \ov Y\,\right] \;.
$$
In order to lift it to $T\du T'$ one has to choose a point $z\in T$ with $\ov
z=\ov{x'\cw y'}$ and such that $z$ is a descendant of $x$ (i.e., $x$ lies on
the geodesic ray $[z,\g)$), and a point $z'\in T'$ with $\ov{z'}=\ov{x\cw y}$
and such that $z'$ is a descendant of $y'$. Then the resulting path
$\F=(\f,\f')$ with the projections
$$
\f = [x,z]\, [z,x\cw y]\, [x\cw y,y] \;, \qquad \f' = [x',x'\cw y']\, [x'\cw
y',z']\, [z',y']
$$
is a geodesic joining $X$ and $Y$, and all geodesics between $X$ and $Y$ have
this form, see \figref{fig:geod2}.

\begin{figure}[h]
\begin{center}
     \psfrag{x}[cr][cr]{$x$}
     \psfrag{y}[cr][cr]{$y$}
     \psfrag{x1}[cl][cl]{$x'$}
     \psfrag{y1}[cl][cl]{$y'$}
     \psfrag{xy}[cr][cr]{$x\cw y$}
     \psfrag{xy1}[cl][cl]{$x'\cw y'$}
     \psfrag{z}[cr][cr]{$z$}
     \psfrag{z1}[cl][cl]{$z'$}
          \includegraphics{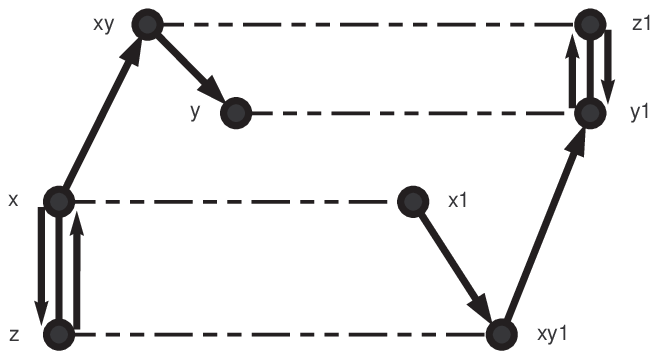}
          \end{center}
          \caption{}
          \label{fig:geod2}
\end{figure}

\medskip

(ii) $\ov X>\ov Y$. \emph{Mutatis mutandis}, the situation is precisely the
same as in case (i), see \figref{fig:geod3}.

\begin{figure}[h]
\begin{center}
     \psfrag{x}[cr][cr]{$x$}
     \psfrag{y}[cr][cr]{$y$}
     \psfrag{x1}[cl][cl]{$x'$}
     \psfrag{y1}[cl][cl]{$y'$}
     \psfrag{xy}[cr][cr]{$x\cw y$}
     \psfrag{xy1}[cl][cl]{$x'\cw y'$}
     \psfrag{z}[cr][cr]{$z$}
     \psfrag{z1}[cl][cl]{$z'$}
          \includegraphics{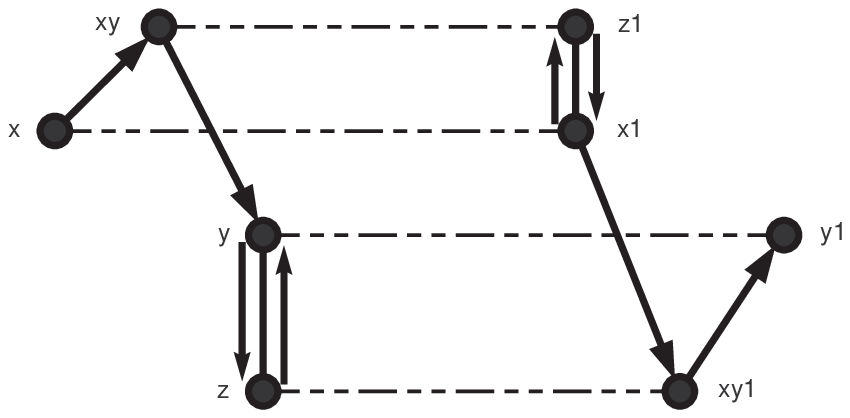}
          \end{center}
          \caption{}
          \label{fig:geod3}
\end{figure}

\medskip

(iii) $\ov X=\ov Y$. In this case, due to absence of the $\left[\,\ov X,\ov
Y\,\right]$ segment, there are two paths in $\Z$ satisfying condition
\eqref{eq:proj}:
$$
\phi_1 = \left[\,\ov X, \ov{x'\cw y'}\,\right] \, \left[\,\ov{x'\cw y'},
\ov{x\cw y}\,\right] \, \left[\,\ov{x\cw y}, \ov Y\,\right] \;.
$$
and
$$
\phi_2 = \left[\,\ov X, \ov{x\cw y}\,\right] \, \left[\,\ov{x\cw y}, \ov{x'\cw
y'}\,\right] \, \left[\,\ov{x'\cw y'}, \ov Y\,\right] \;.
$$
Correspondingly, there are two types of geodesics joining $X$ and $Y$, see
\figref{fig:geod4}.

\begin{figure}[h]
\begin{center}
     \psfrag{x}[cr][cr]{$x$}
     \psfrag{y}[bl][bl]{$y$}
     \psfrag{x1}[tr][tr]{$x'$}
     \psfrag{y1}[cl][cl]{$y'$}
     \psfrag{xy}[cr][cr]{$x\cw y$}
     \psfrag{xy1}[cl][cl]{$x'\cw y'$}
     \psfrag{z}[cr][cr]{$z_1$}
     \psfrag{z1}[cl][cl]{$z'_1$}
     \psfrag{w}[cr][cr]{$z_2$}
     \psfrag{w1}[cl][cl]{$z'_2$}
          \includegraphics{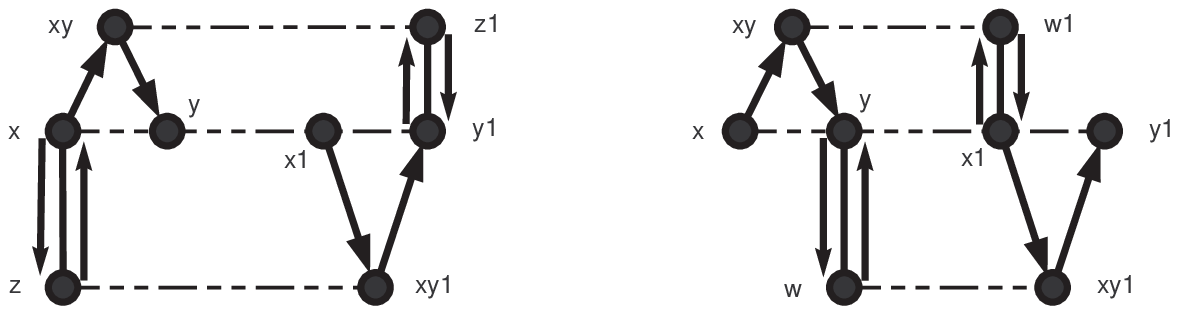}
          \end{center}
          \caption{}
          \label{fig:geod4}
\end{figure}

\end{proof}

By letting the length of geodesics go to infinity in the classification
obtained in the proof of \propref{pr:distance}, we obtain the following
description of geodesic rays and bilateral geodesics in $T\du T'$:

\begin{prop} \label{pr:ray}
Given a point $X=(x,x')\in T\du T'$, any pair $(z,\o')\in T\times \pt_\odot
T'$ with $\ov z = \ov{x'\cw\o'}$ determines a geodesic ray $\F=(\f,\f')$ in
$T\du T'$ issued from $(x,x')$ with the projections
$$
\f = [x,z]\, [z,\g) \;, \qquad \f'=[x',\o') \;;
$$
any pair $(\o,z')\in \pt_\odot T\times T'$ with $\ov{z'} = \ov{x\cw\o}$
determines a geodesic ray $\F=(\f,\f')$ in $T\du T'$ issued from $(x,x')$ with
the projections
$$
\f = [x,\g) \;, \qquad \f'=[x',z']\, [z',\o') \;,
$$
and all geodesic rays in $T\du T'$ are of this form, see \figref{fig:georay}.
\end{prop}

\begin{figure}[h]
\begin{center}
     \psfrag{x}[cr][cr]{$x$}
     \psfrag{x1}[tr][tr]{$x'$}
     \psfrag{xy}[cr][cr]{$x\cw \o$}
     \psfrag{xy1}[cl][cl]{$x'\cw \o'$}
     \psfrag{z}[cr][cr]{$z$}
     \psfrag{z1}[cl][cl]{$z'$}
     \psfrag{g}[b][b]{$\g$}
     \psfrag{o1}[b][b]{$\o'$}
     \psfrag{g1}[t][t]{$\g'$}
     \psfrag{o}[t][t]{$\o$}
          \includegraphics{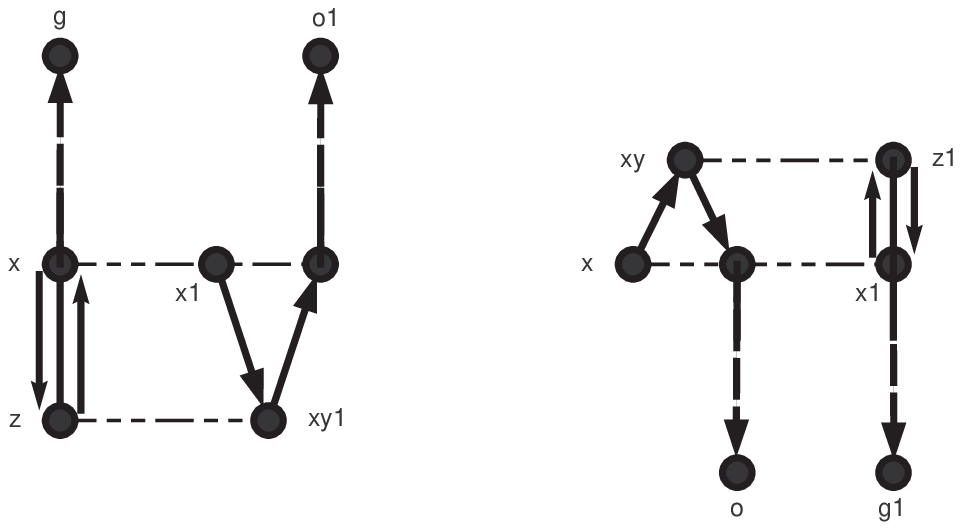}
          \end{center}
          \caption{}
          \label{fig:georay}
\end{figure}

\begin{prop} \label{pr:bil}
All bilateral geodesics $\F$ in $T\du T'$ belong to one of the following 3
classes described in terms of their projections $\phi$ to $\Z$ (by the height
functions) and $\f,\f'$ to $T,T'$, respectively:
\begin{itemize}
\item[(i)]
$\phi$ coincides with $\Z$ (run in either positive or negative direction), and
$\f,\f'$ are, respectively, the bilateral geodesics $(\o,\g)$ and $(\g',\o')$
(both run in either the positive or the negative direction) with
$\g\in\pt_\odot T$ and $\g'\in\pt_\odot T'$;
\item[(ii)]
There is $h\in\Z$ such that $\phi$ is the concatenation $(-\infty,h]
[h,-\infty)$ of two copies of the geodesic ray $[h,-\infty)$ run in the
opposite directions, $\f$ is the geodesic $(\o_1,\o_2)$ for certain
$\o_1\neq\o_2\in\pt_\odot T$ with $\ov{\o_1\cw\o_2}=h$, and
$\f'=(\g',x'][x',\g')$ for a certain $x'\in T'$ with $\ov{x'}=h$;
\item[(iii)]
The same as (ii) with $T$ and $T'$ exchanged: $\phi=(\infty,h][h,\infty)$,
$\f=(\g,x][x,\g)$ for $x\in T$ with $\ov x=h$, and $\f'=(\o'_1,\o'_2)$ for
$\o'_1\neq\o'_2\in\pt_\odot T'$ with $\ov{\o'_1\cw\o'_2}=h$.
\end{itemize}
\end{prop}

\subsection{Regular sequences} \label{sec:reg}

\begin{defn}
A sequence of points $(x_n)$ in a connected graph $X$ is called \emph{regular}
if there exist a geodesic ray $\F$ (with the natural parameterization) and a
real number $\l\ge 0$ (the \emph{rate of escape}) such that
$$
d(x_n,\F(\l n)) = o(n) \;.
$$
If $\l=0$, then $(x_n)$ is called a \emph{trivial} regular sequence.
\end{defn}

This notion was introduced by Kaimanovich \cite{Kaimanovich89} by analogy with
the notion of Lyapunov regularity for sequences of matrices. Any non-trivial
regular sequence in a tree $T$ converges to a boundary point in the
compactification $\ov T= T \sqcup \pt T$ (e.g., see
\cite{Cartwright-Kaimanovich-Woess94}).

\begin{thm} \label{th:ray}
For a sequence of points $X_n=(x_n,x'_n)$ in the horospheric product of trees
$T\du T'$ the following conditions are equivalent:
\begin{itemize}
    \item [(i)]
The sequence $(X_n)$ is regular with the rate of escape $\l\ge 0$;
    \item [(ii)]
$d(X_n,X_{n+1})=o(n)$ and $\ov{X_n}= hn +o(n)$ for a constant (which we call
\emph{height drift}) $h$ with $|h|=\l$;
    \item [(iii)]
The sequences $(x_n)$ and $(x_n')$ are regular in the trees $T$ and $T'$,
respectively, with the same rate of escape $\l$.
\end{itemize}
\end{thm}

\begin{proof}
(i)$\implies$(ii). Obvious in view of inequalities \eqref{eq:dineq} and the
description of geodesic rays in $T\du T'$ from \propref{pr:ray}.

\medskip

(ii)$\implies$(iii). By \eqref{eq:dineq} condition (ii) for the sequence
$(X_n)$ implies that the analogous condition is satisfied for its projections
$(x_n)$ and $(x'_n)$ to the trees $T$ and~$T'$, respectively, i.e.,
$$
d(x_n,x_{n+1})=o(n)\;, \quad d(x'_n,x'_{n+1})=o(n)\;, \quad
\ov{x_n}=\ov{x'_n}=hn+o(n) \;.
$$
Then by \cite[Proposition~1]{Cartwright-Kaimanovich-Woess94} the sequences
$(x_n),(x'_n)$ are both regular with the rate of escape $|h|$.

\medskip

(iii)$\implies$(i). If $\ell=0$, then $(X_n)$ is a trivial regular sequence by
formula \eqref{eq:dDL}. If $\ell>0$, then both $(x_n)$ and $(x'_n)$ are
non-trivial regular sequences. Since $\ov{x_n}=\ov{x'_n}$, one of these
sequences converges to the distinguished boundary point of the corresponding
tree, whereas the other sequence converges to a ``plain'' boundary point. For
instance, let $\lim x_n=\g$ and $\lim x'_n=\o'\in\pt_\odot T'$ (which
corresponds to positivity of the height drift $h$). Take the geodesic ray $\F$
in $T\du T'$ with the projections
$$
\f = [o,z]\, [z,\g) \;, \qquad \f'=[o',\o') \;,
$$
where $\ov z = \ov{o'\cw\o'}$ (cf. \propref{pr:ray}), then $d(X_n,\F(\l
n))=o(n)$.
\end{proof}

\subsection{Boundaries of horospheric products} \label{sec:bdry}

For the horospheric product $T\du T'$ there is a natural compactification
\begin{equation} \label{eq:comp}
\ov{T\du T'}=T\du T'\cup\pt (T\du T')
\end{equation}
obtained by embedding $T\du T'$ into the product $T\times T'$ and further
taking the closure in $\ov T\times \ov{T'}$, where $\ov T$ and $\ov{T'}$ are
the canonical compactifications of the trees $T$ and $T'$, respectively. One
can easily check (see \cite[Proposition 3.2]{Bertacchi01} for details) that
the boundary of this compactification is
$$
\pt(T\du T') = \left(\{\g\}\times{\ov{T'}}\right) \cup \left({\ov
T}\times\{\g'\}\right) \;.
$$
Let
$$
\pt_\uparrow (T\du T') = \{\g\}\times \pt_\odot T' \subset \pt(T\du T')
$$
and
$$
\pt_\downarrow (T\du T') = \pt_\odot T\times \{\g'\} \subset \pt(T\du T')
$$
be, respectively, the \emph{upper} and the \emph{lower} boundaries of the
horospheric product $T\du T'$. Similar pairs of boundaries arise for the
dyadic-rational affine group \cite{Kaimanovich91} or for \emph{treebolic
spaces} \cite{Bendikov-SaloffCoste-Salvatori-Woess11}.

\propref{pr:ray} and \thmref{th:ray} imply

\begin{prop} \label{pr:conv}
A non-trivial regular sequence in $T\du T'$ converges in the compactification
\eqref{eq:comp} either to a point from $\pt_\uparrow (T\du T')$ (if the height
drift is positive) or to a point from $\pt_\downarrow (T\du T')$ (if the
height drift is negative).
\end{prop}

\begin{rem}
It is \emph{not} true (unlike in the tree case) that \emph{any} boundary point
is the limit of a certain non-trivial regular sequence. It might be an
instructive exercise to look at the \emph{Busemann compactification} of the
horospheric product $T\du T'$ (which should not be difficult in view of the
explicit descriptions of geodesics in $T\du T'$ obtained in
\secref{sec:geod}).
\end{rem}

\propref{pr:bil} describes which pairs of boundary points can be joined with a
bilateral geodesic in $T\du T'$ (which is necessarily unique, as it follows
from \propref{pr:bil}). In particular,

\begin{cor} \label{cor:bil}
For any pair of boundary points from $\pt_\downarrow (T\du
T')\times\pt_\uparrow (T\du T')$ there exists a unique bilateral geodesic in
$T\du T'$ joining these points.
\end{cor}

\section{Random horospheric products} \label{sec:random}

\subsection{Graphed equivalence relations and random graphs} \label{sec:eq}

In the present article we shall consider random graphs from the point of view
of the theory of graphed measured equivalence relations. Let us remind the
basic definitions (see \cite{Feldman-Moore77,Adams90,Kaimanovich97}).

Let $(\X,\mu)$ be a Lebesgue measure space (below all the properties related
to measure spaces will be understood \emph{mod~0}, i.e., up to measure 0
subsets). A \emph{partial transformation} of $(\X,\mu)$ is a measure class
preserving bijection between two measurable subsets of $\X$. An equivalence
relation $R\subset\X\times\X$ is called \emph{discrete measured} if it is
generated by an at most countable family of partial transformations. Then
there exists a multiplicative \emph{Radon--Nikodym cocycle}
$\D=\D_\mu:R\to\R_+$ such that for any partial transformation $f:A\to B$ whose
graph is contained in $R$
$$
\D(x,y) = \frac{d\,f^{-1}\mu}{d\mu} (x) = \frac{d\mu}{d\,f\mu}(y) \;.
$$
Alternatively, the Radon--Nikodym cocycle can be defined as the Radon--Nikodym
ratio of the left and the right counting measures on $R$:
$$
\D(x,y) = \frac{d\ch\Mu}{d\Mu} (x,y) \;,
$$
where the \emph{left counting measure} $\Mu$ on $R$ is the result of
integration of the counting measures $\Nu_x$ on the classes of the equivalence
relation (considered as the fibers of the \emph{left projection} $\pi:(x,y)\to
x$ from $R$ onto $\X$) against the measure $\mu$ on the state space $\X$:
$$
d\Mu(x,y) = d\mu(x) d\Nu_x(y) = d\mu (x) \;,
$$
and the \emph{right counting measure} $\ch\Mu$ is the image of the left one
under the \emph{involution}
$$
(x,y)\mapsto (y,x) \;.
$$
If the Radon--Nikodym cocycle $\D$ is identically 1, then the measure $\mu$ is
called \emph{$R$-invariant} ($\equiv$ the equivalence relation $R$
\emph{preserves} the measure $\mu$).

A (non-oriented) \emph{graph structure} on a discrete measured equivalence
relation $(\X,\mu,R)$ is determined by a measurable symmetric subset $K\subset
R\setminus\diag$. The result of the restriction of this graph structure to an
equivalence class $[x]$ gives the \emph{leafwise graph} denoted by~$[x]^K$ (by
analogy with the theory of foliations classes of a discrete equivalence
relation are often called \emph{leaves}). We shall call $(\X,\mu,R,K)$ a
\emph{graphed equivalence relation}. We shall always deal with the graph
structures which are \emph{locally finite}, i.e., any vertex has only finitely
many neighbours, and denote by $\deg$ the integer valued function which
assigns to any point $x\in \X$ the degree (valency) of $x$ in the graph
$[x]^K$. We shall also always assume that the graph structure is
\emph{leafwise connected}, i.e., a.e.\ leafwise graph $[x]^K$ is connected.
Let us denote by $[x]^K_\bullet=([x]^K,x)$ the graph $[x]^K$ rooted at the
point $x$. Thus, we have the map $x \mapsto [x]^K_\bullet$ from $\X$ to the
space of connected locally finite rooted graphs $\GG$ endowed with the usual
ball-wise convergence topology. In particular, if $\mu$ is a probability
measure, then its image under the above map is a probability measure on the
space of rooted graphs $\GG$, i.e., a \emph{random rooted graph}.

\subsection{Random walks on equivalence relations} \label{sec:rweq}

\begin{defn}[\cite{Kaimanovich98}]
A \emph{random walk along equivalence classes} of a discrete measured
equivalence relation $(\X,\mu,R)$ is determined by a measurable family of
leafwise transition probabilities $\{\pi_x\}_{x\in X}$, so that any $\pi_x$ is
concentrated on the equivalence class of $x$, and
$$
(x,y)\mapsto p(x,y)=\pi_x(y)
$$
is a measurable function on $R\subset\X\times\X$. By
$$
p^n(x,y) = \pi_x^n(y) \;, \qquad n\ge 1 \;,
$$
we shall denote the corresponding $n$-step transition probabilities which are
then also measurable as a function on $R$.
\end{defn}

Since the measure class of $\mu$ is preserved by the equivalence relation $R$,
the associated Markov operator $P$ on the space $L^\infty(\X,\mu)$ is
well-defined (cf. \propref{pr:deriv} below). The dual operator then acts on
the space of measures $\la$ absolutely continuous with respect to $\mu$
(notation: $\la\prec\mu$). Following a probabilistic tradition, we shall
denote this action by $\la\mapsto\la P$. The density of the measure $\la P$
with respect to $\mu$ can be explicitly described in terms of the density of
$\la$ and of the Radon--Nikodym cocycle $\D=\D_\mu$ of the measure $\mu$.

\begin{prop}[\cite{Kaimanovich98}] \label{pr:deriv}
For any $\si$-finite measure $\la\prec\mu$
\begin{equation} \label{eq:dens}
\frac{d\la P}{d\mu}(y) = \sum_{x\in [y]} p(x,y) \, \D(y,x) \,
\frac{d\la}{d\mu}(x) \;.
\end{equation}
\end{prop}

\begin{proof}
Let us run the Markov chain determined by the operator $P$ from the initial
(time~0) distribution $\la$. Then the time 1 distribution is, by definition,
the measure $\la P$, and the joint distribution of the positions of the chain
at times 0 and~1 is the measure
\begin{equation} \label{eq:Pi}
d\Pi(x,y) = d\la(x) \, p(x,y) \;,
\end{equation}
which is obviously absolutely continuous with respect to the counting measure
$\Mu$. The corresponding Radon--Nikodym derivative is
\begin{equation} \label{eq:trans}
\frac{d\Pi}{d\Mu}(x,y) = \frac{d\la}{d\mu}(x)\, p(x,y) \;.
\end{equation}
Since the left and the right counting measures are equivalent,
\begin{equation} \label{eq:tr}
\frac{d\Pi}{d\ch\Mu}(x,y) = \frac{d\Pi}{d\Mu}(x,y) \,
\frac{d\Mu}{d\ch\Mu}(x,y) = \frac{d\la}{d\mu}(x)\, p(x,y) \, \D(y,x) \;.
\end{equation}
On the other hand, since $\la P$ is the result of the right projection of the
measure $\Pi$ to $\X$,
\begin{equation} \label{eq:cotr}
\frac{d\Pi}{d\ch\Mu}(x,y) = \frac{d\la P}{d\mu}(y)\,\ch p(y,x) \;,
\end{equation}
where $\ch p(\cdot,\cdot)$ are the corresponding \emph{cotransition
probabilities} (cf. formula \eqref{eq:trans}), whence summation of
\eqref{eq:tr} over $x\in [y]$ yields the claim.
\end{proof}

\begin{rem}
If the measure $\la$ is infinite, then the density from formula
\eqref{eq:dens} may well be infinite on a set of positive measure $\mu$;
however, even in this case the measure $\la P$ is absolutely continuous with
respect to $\mu$ in the sense that any null set of $\mu$ is also null with
respect to $\la P$.
\end{rem}

Comparison of \eqref{eq:trans} and \eqref{eq:cotr} leads to the following
useful formula relating transition and cotransition probabilities:
$$
\frac{d\la}{d\mu}(x)p(x,y) = \frac{d\la P}{d\mu}(y)\ch p(y,x) \D(x,y) \;,
$$
or, in a somewhat informal way,
$$
d\la(x) p(x,y) = d\la P(y) \ch p(y,x) \;,
$$
which is what one could expect.

Given a measure $\la\prec\mu$, we denote by $\{\la_x\}_{x\in\X}$ the family of
leafwise measures on the equivalence classes of $\X$ defined as
\begin{equation} \label{eq:la}
\la_x(y) = \frac{d\la(y)}{d\mu(x)} = \frac{d\la}{d\mu}(y) \D(x,y) \;.
\end{equation}
The measures $\la_x$ corresponding to different equivalent points $x$ are
obviously all proportional.

\begin{cor} \label{cor:stat}
A measure $\la\prec\mu$ is $P$-stationary (i.e., $\la=\la P$) if and only if
the leafwise measures $\la_x$ \eqref{eq:la} are almost surely stationary with
respect to the transition probabilities $p(\cdot,\cdot)$.
\end{cor}

If the measure $\la$ is stationary, then, as it follows from a comparison of
formulas \eqref{eq:tr} and \eqref{eq:cotr}, the cotransition probabilities
$$
\ch p(y,x) = p(x,y) \D_\la(y,x) \,\;,
$$
where $\D_\la$ is the Radon--Nikodym cocycle of the measure $\la$, determine
another Markov chain along equivalence classes with the same stationary
measure $\la$. It is called the \emph{time reversal} of the original random
walk. If it coincides with the original walk, then the latter is called
\emph{reversible}.

If the equivalence relation $(\X,\mu,R)$ is endowed with a graph structure
$K$, then the transition probabilities
$$
p(x,y) = \left\{%
\begin{array}{ll}
    1/\deg x\;, & \hbox{$(x,y)\in K$\;} \\
    0\;, & \hbox{otherwise\;.} \\
\end{array}%
\right.
$$
determine the \emph{simple random walk} along the classes of the equivalence
relation $R$.

\begin{cor} \label{cor:simple}
Let $(\X,\mu,R,K)$ be a graphed equivalence relation. If the measure $\mu$ is
$R$-invariant, then the measure $\la=\deg\cdot\mu$ is stationary with respect
to the simple random walk along the classes of $R$ (i.e., $\la=\la P$, where
$P$ is the Markov operator of the simple random walk).
\end{cor}

\begin{rem}
Actually, $R$-invariance of the measure $\mu$ is precisely equivalent to the
combination of the above stationarity condition with \emph{reversibility} of
the leafwise simple random walk with respect to the measure $\deg\cdot\mu$,
see \cite[Proposition 2.4.1 and its Corollary]{Kaimanovich98}.
\end{rem}

\subsection{Entropy and leafwise Poisson boundaries} \label{sec:entr}

Below we shall be interested in describing the \emph{Poisson boundary} of
Markov chains along individual classes of an equivalence relation. We remind,
without going into details, that the Poisson boundary is responsible for
describing the stochastically significant behaviour of a Markov chain at
infinity. The main tools used for its identification are the general
\emph{0--2 laws} and the \emph{entropy theory} (see \cite{Kaimanovich92} and
the references therein). The latter one is especially expedient when dealing
with random walks on groups \cite{Kaimanovich-Vershik83,Kaimanovich00a}.

As it was on numerous occasions mentioned by the first author (e.g., see
\cite{Kaimanovich86,Kaimanovich88,Kaimanovich90a,Kaimanovich98}), the entropy
theory is also applicable in all situations when there is an appropriate
probability path space endowed with a measure preserving time shift. The most
general currently known setup is provided by \emph{random walks on groupoids}
\cite{Kaimanovich05a} with a finite stationary measure on the space of
objects. A particular case of it consists of Markov chains along classes of an
equivalence relation in the presence of a global stationary probability
measure \cite{Kaimanovich98}, and, as it was pointed out in
\cite{Kaimanovich98}, the entropy theory is perfectly applicable in this
situation, providing criteria for triviality and identification of the Poisson
boundary of leafwise random walks (also see \cite{Kaimanovich03a} where this
theory was used for describing the Poisson boundary on the graphed equivalence
relations associated with the fractal limit sets of iterated function
systems). In other particular cases (most of which can actually be completely
described in terms of random walks on equivalence relations) the entropy
theory was along the same lines implemented in \cite{Kaimanovich-Woess02} (for
Markov chains with a transitive group of symmetries),
\cite{AlcaldeCuesta-FernandezdeCordoba11} (for random walks along orbits of
pseudogroups acting on a measure space), \cite{Bowen10p} (for random walks on
random Schreier graphs), \cite{Benjamini-Curien10p} (for simple random walks
on unimodular and stationary random graphs).

Since \cite{Kaimanovich98} and \cite{Kaimanovich03a} contain only a brief
outline of the entropy theory for random walks along equivalence relations, we
shall give more details here (although these arguments are essentially the
same as in the case of random walks in random environment
\cite{Kaimanovich90a}). For the rest of this section we shall assume that
$(\X,\mu,R)$ is a discrete measured equivalence relation endowed with a Markov
operator $P$ determined by a measurable family of transition probabilities
$\pi_x$, and that $\la\prec\mu$ is a $P$-stationary probability measure.
Denote by $\P_x$ the probability measure in the space of paths of the
associated leafwise Markov chain issued from a point $x\in\X$. The
one-dimensional distributions of $\P_x$ are the $n$-step transition
probabilities $\pi_x^n$ from the point $x$.

\begin{thm} \label{th:02}
For $\la$-a.e.\ point $x\in\X$ the tail and the Poisson boundaries of the
leafwise Markov chain coincide $\P_x$ -- \textup{mod 0}.
\end{thm}

\begin{proof}
Let
$$
\f_n(x) = \| \d_x P^n - \d_x P^{n+1} \| = \| \pi_x^n - \pi_x^{n+1} \|
\;,\qquad x\in X \;.
$$
Then
$$
\begin{aligned}
\f_{n+1}(x) = \| \d_x P^{n+1} - \d_x P^{n+2} \| = \| (\d_x P^n - \d_x
P^{n+1})P \|& \\
\le \| \d_x P^n - \d_x P^{n+1} \|& = \f_n(x) \;,
\end{aligned}
$$
so that there exists a limit
$$
\f(x) = \lim_n \f_n(x) \;.
$$
Moreover,
$$
\begin{aligned}
\f_{n+1}(x) &= \| \d_x P^{n+1} - \d_x P^{n+2} \| = \| \d_x P(P^n-P^{n+1}) \| \\
&\le \sum_y p(x,y) \| \d_y P^n - \d_y P^{n+1} \| = \sum_y p(x,y) \f_n(y) \;,
\end{aligned}
$$
whence $\f$ is \emph{subharmonic}:
$$
\f \le P\f \;.
$$
The function $\f$ is clearly measurable. Then
$$
\langle \la,P\f \rangle = \langle \la P, \f \rangle = \langle \la, \f \rangle
$$
by stationarity of the measure $\la$, so that in fact $\f$ is harmonic.
Therefore, by a classical property of Markov chains with a finite stationary
measure $\f$ must be constant along a.e.\ sample path (e.g., see
\cite{Kaimanovich92}). By the corresponding 0--2 law (see again
\cite{Kaimanovich92}) in this situation $\f$ can take values 0 and 2 only
(obviously, in the ergodic case only one of these two options may occur). In
the first case the Poisson and the tail boundary coincide for an arbitrary
initial distribution, whereas in the second case for any $x\in\X$ the
one-dimensional distributions $\pi_x^n$ are all pairwise singular, so that the
Poisson and the tail boundaries coincide $\P_x$ -- mod 0.
\end{proof}

Let
$$
H_n(x)=H(\pi_x^n)
$$
be the \emph{entropies} of $n$-step transition probabilities, and let
\begin{equation} \label{eq:ent}
H_n = \int H_n(x)\,d\la(x)
\end{equation}
be their averages over the space $(\X,\la)$. In terms of the shift-invariant
measure
$$
\P_\la=\int\P_x\,d\la(x)
$$
on the space of sample paths $\x=(x_n)\in\X^{\Z_+}$ which corresponds to the
stationary initial distribution $\la$,
\begin{equation} \label{eq:entform}
H_n = - \int \log \pi^n_{x_0}(x_n) \,d\P_\la(\x) \;.
\end{equation}
In yet another language, that of measurable partitions and their (conditional)
entropies (e.g., see \cite{Rohlin67}),
\begin{equation} \label{eq:entcond}
H_n = \H_\la (\a_n|\a_0) = \int \H_x(\a_n)\,d\la(x) \;,
\end{equation}
where $\a_k$ denotes the $k$-th coordinate partition in the path space
$\X^{\Z_+}$, and $\H_\la$ (resp., $\H_x$) denotes the (conditional) entropy
with respect to the measure $\P_\la$ (resp.,~$\P_x$).

\begin{thm} \label{th:entr}
If $H_1<\infty$, then all the average entropies $H_n$ are also finite, there
exists a limit (the \emph{asymptotic entropy})
\begin{equation} \label{eq:asent}
\h=\h(P,\la) = \lim_n \frac {H_n}{n} < \infty \;,
\end{equation}
and $\h=0$ if and only if for $\la$-a.e. point $x\in\X$ the Poisson boundary
of the leafwise Markov chain is trivial with respect to the measure $\P_x$.
\end{thm}

Let us put
$$
\a_k^n = \bigvee_{i=k}^n \a_i \;,\qquad 0\le k \le n \le \infty \;,
$$
where, as before, $\a_i$ are the coordinate partitions in the path space. For
proving \thmref{th:entr} we shall need the following

\begin{lem} \label{lem:entr}
For any $k\le n$ the conditional entropy of the partition $\a_1^k$ with
respect to the partition $\a_n^\infty$ in the path space
$\left(\X^{\Z_+},\P_\la\right)$ is
\begin{equation} \label{eq:condentr}
\H_\la\left( \a_1^k \,|\, \a_0\vee \a_n^\infty \right) = \H_\la\left( \a_1^k
\,|\, \a_0\vee\a_n \right) = k H_1 + H_{n-k} - H_n \;.
\end{equation}
\end{lem}

\begin{proof}
Formula \eqref{eq:condentr} and its proof are completely analogous to the
group case considered in \cite{Kaimanovich-Vershik83}. Indeed, the leftmost
identity in \eqref{eq:condentr} immediately follows from the Markov property,
whereas
$$
\H_\la\left( \a_1^k \,|\, \a_0\vee\a_n \right) = \int \H_x \left( \a_1^k \,|\,
\a_n \right)\,d\la(x) \;,
$$
cf. formula \eqref{eq:entcond}. By the definition of conditional entropy, for
any $x\in\X$
$$
\H_x \left( \a_1^k \,|\, \a_n \right) = - \int \log \P_x \left( \a_1^k(\x)
\,|\, \a_n(\x) \right)\,d\P_x(\x) \;,
$$
where $\xi(\x)$ denotes the element of a partition $\xi$ which contains a
sample path $\x$. Now,
\begin{equation} \label{eq:telescope}
\begin{aligned}
\P_x \left( \a_1^k(\x) \,|\, \a_n(\x) \right)
 &= \frac{\P_x \left( \a_1^k(\x) \cap \a_n(\x) \right)}{\P_x \left( \a_n(\x) \right)} \\
 &= \frac{p(x_0,x_1)p(x_1,x_2)\cdots p(x_{k-1},x_k) p^{n-k}(x_k,x_n) }{p^n(x_0,x_n)} \;,
\end{aligned}
\end{equation}
which implies the claim in view of formula \eqref{eq:entform} and shift
invariance of the measure~$\P_\la$.
\end{proof}

\begin{proof}[Proof of \thmref{th:entr}]
Formula \eqref{eq:entform} in combination with shift invariance of the measure
$\P_\la$ easily implies subadditivity of the sequence $H_n$ and existence of
the limit \eqref{eq:asent}. Moreover, the sequence of functions
$\f_n(\x)=-\log\pi_{x_0}^n(x_n)$ satisfies conditions of Kingman's subadditive
ergodic theorem, which implies existence of individual limits
$$
\lim_n -\frac1n \log\pi_{x_0}^n(x_n)
$$
for $\P_\la$-a.e.\ sample path $\x=(x_n)$ as well. If the shift $T$ is
ergodic, then these individual limits almost surely coincide with $\h$. Note
that ergodicity of $T$ is equivalent to absence of non-trivial subsets of the
state space $\X$ invariant with respect to the operator $P$ (by aforementioned
general property of Markov chains with a finite stationary measure), which, in
the case when pairs of points $(x,y)\in R$ with $\pi_x(y)>0$ generate the
relation $R$, is equivalent to ergodicity of $R$.

Actually, \lemref{lem:entr} provides a stronger form of existence of the limit
\eqref{eq:asent}. Namely, since the sequence of partitions
$\a_0\vee\a_n^\infty$ is decreasing on $n$, monotonicity properties of
conditional entropy (e.g., see \cite{Rohlin67}) imply that $\H_\la\left( \a_1
\,|\, \a_0\vee \a_n^\infty \right)$ increases on $n$. In view of formula
\eqref{eq:condentr} it means that not only the limit $\h=\lim H_n/n$ exists,
but also that $\bigl[H_{n+1}-H_n\bigr]\searrow\h$.

As we have already seen on a similar occasion in the proof of
\lemref{lem:entr}, the left-hand side of formula \eqref{eq:condentr} can be
rewritten as
$$
\H_\la\left( \a_1^k \,|\, \a_0\vee \a_n^\infty \right) = \int \H_x \left(
\a_1^k \,|\, \a_n^\infty \right) \,d\la(x) \;.
$$
By continuity of conditional entropy (see again \cite{Rohlin67}), for any
$x\in\X$
$$
\H_x \left( \a_1^k \,|\, \a_n^\infty \right) \nearrow \H_x \left( \a_1^k \,|\,
\a^\infty \right) \le \H_x \left( \a_1^k \right) \;,
$$
where $\a^\infty=\lim_n\a_n^\infty$ is the \emph{tail partition}. The
right-hand side in the above formula is integrable, and
$$
\int \H_x \left( \a_1^k \right) = k H_1 \;,
$$
cf. formula \eqref{eq:telescope}. Therefore, after passing in
\eqref{eq:condentr} to a limit as $n\to\infty$ we conclude that for any $k>0$
$$
\int \H_x \left( \a_1^k \,|\, \a^\infty \right) \,d\la(x) = k (H_1-\h)
$$
and
$$
k\h = \int \Bigl[ \H_x(\a_1^k) - \H_x \left( \a_1^k \,|\, \a^\infty \right)
\Bigr] \,d\la(x) \;.
$$
It means that $\h=0$ if and only if for $\la$-a.e.\ $x\in\X$ the tail
partition $\a^\infty$ is $\P_x$-independent of all coordinate partitions
$\a_1^k$, the latter condition being equivalent to triviality of the tail
partition $\P_x$ -- mod 0.

Finally, by \thmref{th:02}, for $\P_\la$-a.e.\ $x\in\X$ the tail and the
Poisson boundaries coincide $\P_x$ -- mod 0, which completes the proof.
\end{proof}

By passing to an appropriate \emph{boundary extension} of the original
equivalence relation \cite{Kaimanovich05a}, \thmref{th:entr} is also
applicable to the problem of description of non-trivial Poisson boundaries of
leafwise Markov chains. Indeed, a quotient of the Poisson boundary is
\emph{maximal} (i.e., coincides with the whole Poisson boundary) if and only
if for almost all conditional chains determined by the points of this quotient
the Poisson boundary is trivial. Thus, the criterion from \thmref{th:entr}
allows one to carry over the \emph{ray} and the \emph{strip} criteria used for
identification of the Poisson boundary in the group case \cite{Kaimanovich00a}
to the setup of random walks along classes of graphed equivalence relations.

\subsection{The Poisson boundary of random horospheric products}
\label{sec:final}

Below by a \emph{random horospheric product} we shall mean a graphed
equivalence relation $(\X,\mu,R,K)$ such that a.e.\ leafwise graph is a
horospheric product. Moreover, we shall assume that the ``orientations''
(signs) of leafwise height cocycles \eqref{eq:height} are chosen in a
consistent way, i.e., that there exists a global $\Z$-valued \emph{measurable}
cocycle $\B$ on $R$ such that its restriction to a.e.\ leaf is a height
cocycle. Therefore, a.e.\ leafwise graph $[x]^K$, being a horospheric product,
is endowed with its lower and upper boundaries $\pt_\downarrow [x]^K$ and
$\pt_\uparrow [x]^K$, respectively, and one can easily see that the
corresponding boundary bundles over $(\X,\mu,R,K)$ are measurable (cf.
\cite{Kaimanovich04}).

\begin{thm} \label{th:main}
Let $(\X,\mu,R,K,\B)$ be a random horospheric product with uniformly bounded
vertex degrees, and $P$ --- the Markov operator of a random walk along classes
of the equivalence relation $R$ determined by a measurable family of leafwise
transition probabilities $\{\pi_x\}_{x\in\X}$. If $\la\prec\mu$ is a
$P$-stationary probability measure such that the transition probabilities
$\{\pi_x\}$ have a finite first moment
\begin{equation} \label{eq:mom}
\int_R d(x,y)\,d\Pi(x,y) \;,
\end{equation}
where $d$ is the leafwise graph distance, and $\Pi$ is the measure
\eqref{eq:Pi}, then the Poisson boundaries of leafwise random walks are
determined by the \emph{global height drift}
$$
h = \int_R \B(x,y)\,d\Pi(x,y) \;.
$$
If $h=0$, then the Poisson boundary is a.s.\ trivial, whereas when $h>0$
(resp., $h<0$) a.e. leafwise Poisson boundary coincides (mod 0) with the upper
(resp., lower) leafwise boundary endowed with the corresponding limit
distribution (which is well-defined by \propref{pr:conv}).
\end{thm}

\begin{proof}
\thmref{th:ray} in combination with the standard ergodic arguments (cf.
\cite{Kaimanovich00a}) implies that a.e.\ sample path is regular with the
height drift $h$. If $h=0$, then regularity implies vanishing of the
asymptotic entropy, and therefore triviality of leafwise Poisson boundaries.
If $h\neq 0$, then by \propref{pr:conv} a.e.\ sample path converges to the
corresponding boundary (the upper, if $h>0$, and the lower, if $h<0$) of
leafwise horospheric products. The fact that these boundaries are actually
maximal (i.e., coincide with the leafwise Poisson boundaries) then follows
from the ray criterion (or \corref{cor:bil} in combination with the strip
criterion) in precisely the same way as in the group case, cf.
\cite{Kaimanovich00a}.
\end{proof}

\begin{rem}
Finiteness of the entropies \eqref{eq:ent} (which is crucial for
\thmref{th:entr}) follows, in the usual way, from finiteness of the first
moment \eqref{eq:mom} and uniform boundedness of vertex degrees, cf. \cite[p.
259]{Derriennic86} or \cite[Lemma~5.2]{Kaimanovich00a}.
\end{rem}

Obviously, if the operator $P$ is reversible with respect to a stationary
measure $\la$ (see the discussion at the end of \secref{sec:rweq} for the
definition), then the integral of any additive cocycle on $R$ with respect to
$\Pi$ vanishes. In particular, in this case the global height drift $h$
vanishes, whence

\begin{cor}
Under conditions of \thmref{th:main}, if the operator $P$ is reversible with
respect to the measure $\la$, then the leafwise Poisson boundaries are a.s.\
trivial.
\end{cor}

\begin{cor}
Under conditions of \thmref{th:main}, if $\la=\deg\cdot\mu$ is the stationary
measure of the leafwise simple random walk corresponding to a finite
$R$-invariant measure $\mu$ (see \corref{cor:simple}), then the Poisson
boundary of the leafwise simple random walks is a.s.\ trivial.
\end{cor}

\bibliographystyle{amsalpha}
\bibliography{C:/Sorted/MyTex/mine}

\enddocument

\bye